\documentclass[oneside]{amsart}
\usepackage{macros}

\title{Examples and Nonexamples of Distal Metric Structures}
\author{Aaron Anderson and Ita\"i Ben Yaacov}
\address{Aaron Anderson \\
  University of Pennsylvania, Department of Mathematics \\
  209 S 33rd St, Philadelphia, PA 19104 \\
  United States}

\urladdr{\url{http://awainverse.github.io}}
\address{Itaï \textsc{Ben Yaacov} \\
  Université Claude Bernard -- Lyon 1 \\
  Institut Camille Jordan, CNRS UMR 5208 \\
  43 boulevard du 11 novembre 1918 \\
  69622 Villeurbanne Cedex \\
  France}

\urladdr{\url{http://math.univ-lyon1.fr/~begnac/}}

\begin{document}

\begin{abstract}
    This article provides examples of distal metric structures.
One source of examples are metric valued fields.
By analyzing indiscernible sequences, we show that real closed metric valued fields are distal, and conclude that
algebraically closed metric valued fields, while stable, have the strong Erd\H{o}s-Hajnal property, which we define appropriately for metric structures.

We find another example in topological dynamics: we study a metric structure whose automorphism group is the well-understood Polish group $\mathrm{Hom}^+([0,1])$ of increasing homeomorphisms of $[0,1]$.
This was known\cite{ibarlucia1} to be NIP and highly unstable, and further properties were established in \cite{no_trans}.
We characterize models of its theory of this structure, which we call \emph{Dual Linear Continua}, up to isomorphism.
We analyze their indiscernible sequences and prove that they are distal, as well as constructing explicit distal cell decompositions.
\end{abstract}

\maketitle

\section{Introduction}
The school of model theory called neostability revolves around classifying first-order structures based on combinatorial properties.
The classes of structures studied include stable structures, such as algebraically closed fields, where no infinite linear orders are definable,
and NIP structures, including the stable structures, whose definable sets are combinatorially well-behaved. 
Distal structures, introduced in \cite{distal_simon}, are a class of first-order structures which are NIP, but given this, are as far as possible from being stable.
As such, orders tend to be central to the study of distal structures.
Frequently the structures themselves are defined as expansions of a linear order, such as $o$-minimal structures and their generalizations (weakly or quasi-$o$-minimal structures).
Examples of these include dense linear orders, Presburger arithmetic, real closed fields, and various other natural structures on $\R$, potentially including much of the analytic structure.\cite{vdd_tame}
Other examples introduce an order through a valuation, where the order on the valued group still ensures that the structure is distal.
These includes several natural structures on the $p$-adics as well as some differential fields of transseries.\cite{distal_simon,distal_val}

Continuous logic provides a better logical framework for studying metric spaces and analytic objects.
The objects it studies, \emph{metric structures}, provide a variety of new examples for model theory.
Various characterizations of distality were extended to the context of continuous logic in \cite{anderson1,anderson2}, so it is natural to ask which known metric structures are distal or have distal expansions.
However, many of the well-understood metric structures are either combinatorially ill-behaved, such as Urysohn space, or are stable.\cite{urysohn, mtfms, rtrees, hanson_approx}.
In particular, ordered metric structures have not been studied, or even defined, outside of the context of ordered real closed metric valued fields.\cite{mvf}
In this paper, we provide examples of distal metric structures and contrast them with notable non-examples observed by James Hanson.

For combinatorial purposes, it often suffices to consider structures which admit distal expansions, which are necessarily NIP.
In classical logic, these structures are known to have the strong Erd\"os-Hajnal property (SEH): definable relations satisfy a particular strengthening of Ramsey's theorem.
No converse is known - that is, it is unclear if SEH, or any other purely combinatorial test, determines whether an NIP structure has a distal expansion,
although so far, the only way an NIP structure has been shown to \emph{not} have a distal expansion is by showing it lacks SEH.

In Section \ref{sec_seh}, we give a statement of SEH for metric structures.
Any metric structure with a distal expansion has this property as an easy corollary of work in \cite{anderson2},
and in Section \ref{sec_vf}, we provide an example of a stable metric structure (any algebraically closed metric valued field) that admits a distal expansion and thus has SEH.
However, we will also see in Section \ref{sec_non} that at least in continuous logic, there is another way of determining that common examples of stable structures do not admit distal expansions.

In Section \ref{sec_vf}, we examine some metric valued fields.
These structures are constructed by taking a field with valuation in $\R_{\geq 0}$, and incorporating the valuation metric into the metric structure.
In \cite{mvf}, theories of algebraically and real closed metric valued fields are developed.
In Section \ref{sec_vf}, we show, using the indiscernible sequence definition, that real closed metric valued fields are distal.
We also show that algebraically closed metric valued fields are interpretable in real closed metric valued fields, from which we conclude that these have the strong Erd\"os-Hajnal property,
although they are stable and thus not distal.

Section \ref{sec_dlc} explores a fundamentally different distal metric theory, which we call \emph{dual linear continua}.
Models of this theory consist of the set of functions from some linear continuum (such as the linear order $[0,1]$)
to $[0,1]$ which are continuous, nondecreasing, and surjective, with a particular structure placed upon them.
In the case of the linear continuum $[0,1]$, the automorphism group of this structure is the group of increasing homeomorphisms from $[0,1]$ to itself.
This structure had been studied before in \cite{no_trans} and in \cite{ibarlucia1}, where it had been shown to be NIP but decidedly not stable.
We show that in fact, the structure is distal, both by studying its indiscernible sequences and by constructing explicit distal cell decompositions.

Finally, in Section \ref{sec_non}, we examine some metric structures which are NIP but do not admit distal expansions, suggested by James Hanson.
These include any metric structure expanding a Banach space, and in particular the Keisler randomization of any (metric) structure.
This shows that unlike stability \cite{random09} or NIP \cite{randomVC}, distality is \emph{not} preserved by taking randomizations.

For background on continuous logic, we refer to \cite{mtfms}, for further notation and background assumed throughout on distal metric structures in particular, we refer to \cite{anderson1}.

\subsection*{Acknowledgements}
We thank James Hanson for several helpful insights, some of which are detailed in Section \ref{sec_non}.
The first author thanks Artem Chernikov for advising, and was partially supported by the Chateaubriand fellowship, the UCLA Logic Center, the UCLA Dissertation Year Fellowship, and NSF grants DMS-1651321 and DMS-2246598.
The second author was partially supported by ANR project AGRUME (ANR-17-CE40-0026).

\section{Strong Erd\H{o}s-Hajnal}\label{sec_seh}
In addition to considering examples of distal metric structures, we will identify interesting reducts of distal metric structures.
Even if these reducts are no longer distal, they will retain properties such as the \emph{strong Erd\H{o}s-Hajnal property}.
In \cite{distal_reg}, it was shown that distality is equivalent to the \emph{definable} strong Erd\H{o}s-Hajnal property, which implies the strong Erd\H{o}s-Hajnal property for all of its reducts.
This characterization of distality was extended to metric structures in \cite{anderson2}, and we will now describe the strong Erd\H{o}s-Hajnal property for reducts of distal metric structures.
We define homogeneity for sets and definable predicates as in \cite{anderson2}:
\begin{defn}
    For $i = 1,\dots,n$, let $A_i \subseteq M^{x_i}$, let $\phi(x_1,\dots,x_n)$ be a definable predicate (possibly with parameters) and let $\varepsilon > 0$.
    Then we say that $(A_i: 1 \leq i \leq n)$ is $(\phi,\varepsilon)$-\emph{homogeneous} when for all $(a_i : i \in I), (a_i' : i \in I) \in A_1\times \dots \times A_n$,
    $|\phi(a_1,\dots,a_n) - \phi(a_1',\dots,a_n')|\leq \varepsilon$.

    If for $1 \leq i \leq n$, $\psi_i(x_i)$ are definable predicates (possibly with parameters), we say that $(\psi_i(x_i) : 1 \leq i \leq n)$ are $(\phi,\varepsilon)$-homogeneous when the supports $\psi_i(x_i) > 0$ are.
\end{defn}

\begin{fact}[{\cite[Corollary 5.7]{anderson2}}]
    A theory $T$ of continuous logic is distal if and only if every definable predicate $\phi(x_1,\dots,x_n;y)$ has the definable strong Erd\H{o}s-Hajnal property:
    
    For every $\varepsilon > 0$, there exist definable predicates $\psi_i(x_i;z_i)$ and $\delta > 0$ such that
    if $\mu_1 \in \mathfrak{M}_{x_1}(M),\dots,\mu_n \in \mathfrak{M}_{x_n}(M)$ are such that for $i < n$, $\mu_i$ is generically stable, and $b \in M^y$, then for any
    product measure $\omega$ of $\mu_1,\dots,\mu_n$, there are $d_i \in M^{z_i}$ such that $\psi_i(x_i;d_i)$ are $(\phi(x;b),\varepsilon)$-homogeneous 
            and $\int_{S_{x_i}(M)}\psi_i(x_i;d_i)\,d\mu_i \geq \delta$ for each $i$.
\end{fact}

As all counting measures are generically stable, we can deduce the following:
\begin{lem}
    Assume $T$ is a distal theory in continuous logic.
    If $\phi(x_1,\dots,x_n)$ is a definable predicate with parameters, then $\phi$ has the strong Erd\H{o}s-Hajnal property:
    Then for every $\varepsilon > 0$, there is some $\delta > 0$ such that if $A_1,\dots,A_n$ are finite subsets of $M^{x_1},\dots,M^{x_n}$ respectively, then there are $B_1,\dots,B_n$
    with $B_i \subseteq A_i$ and $|B_i| \geq \delta |A_i|$ such that $(B_i : 1 \leq i \leq n)$ is $(\phi,\varepsilon)$-homogeneous.
\end{lem}

\begin{lem}
    The strong Erd\H{o}s-Hajnal property is closed under continuous combinations: if for $1 \leq j \leq n$, definable predicates $\phi_j(x) = \phi_j(x_1,\dots,x_m)$ have the strong Erd\H{o}s-Hajnal property,
    and $u : [0,1]^n \to [0,1]$ is continuous, then $u(\phi_1(x),\dots,\phi_n(x))$ also has the strong Erd\H{o}s-Hajnal property.
\end{lem}
\begin{proof}
    Fix $\varepsilon > 0$, and $A_i \subseteq M^{x_i}$ finite for each $1 \leq i \leq m$.
    By continuity, let $\delta > 0$ be such that if $a,b \in [0,1]^n$ have $\max_i |a_i - b_i| \leq \delta$ in the sup metric, then $|u(a) - u(b)| \leq \varepsilon$.
    By the strong Erd\H{o}s-Hajnal property of $\phi_1,\dots,\phi_n$, there is some $\gamma > 0$ such that for each $1 \leq j \leq n$,
    if $B_i \subseteq M^{x_i}$ are finite, there are $C_i \subseteq B_i$ with $|C_i| \geq \gamma|B_i|$ such that $(C_1,\dots,C_m)$ are $(\phi_j,\delta)$-homogeneous.
    Thus we can set $A_i^0 = A_i$, and recursively define $A_i^j$ such that $A_i^j \subseteq A_i^{j - 1}$, $|A_i^j| \subseteq \gamma|A_i^{j - 1}|$, and
    $(A_1^j,\dots,A_m^j)$ is $(\phi_j,\delta)$-homogeneous.
    Then $(A_1^n,\dots,A_m^n)$ will be $(\phi_j,\delta)$-homogeneous for all $1 \leq j \leq n$, and thus also $(u(\phi_1,\dots,\phi_n),\varepsilon)$-homogeneous.
    Also, $|A_i^n| \geq \delta^n|A_i|$, and $\delta$ did not depend on the choice of $A_i$.
\end{proof}

\begin{lem}
    The strong Erd\H{o}s-Hajnal property is closed under uniform limits: if for $j \in \N$, definable predicates $\phi_j(x) = \phi_j(x_1,\dots,x_m)$ have the strong Erd\H{o}s-Hajnal property and converge uniformly to $\phi(x)$,
    then $\phi(x)$ also has the strong Erd\H{o}s-Hajnal property.
\end{lem}
\begin{proof}
    Fix $\varepsilon > 0$.
    Let $N$ be large enough that $\sup_x|\phi_N(x) - \phi(x)| \leq \frac{\varepsilon}{3}$.
    Then let $\delta > 0$ be such that if $B_i \subseteq M^{x_i}$ are finite, there are $C_i \subseteq B_i$ with $|C_i| \geq \delta|B_i|$ such that $(C_1,\dots,C_m)$ are $\left(\phi_N,\frac{\varepsilon}{3}\right)$-homogeneous.
    Then if we fix $A_i \subseteq M^{x_i}$ finite for each $1 \leq i \leq m$,
    there are $B_i \subseteq A_i$ for each $i$ with $|B_i| \geq |A_i|$ and for all $a, b \in B_1 \times \dots \times B_m$,
    we have $|\phi(a) - \phi(b)| \leq |\phi(a) - \phi_N(a)| + |\phi_N(a) - \phi_N(b)| + |\phi_N(b) - \phi(B)| \leq \varepsilon$.
\end{proof}

These lemmas show that in a quantifier-elimination language, to determine if all definable predicates in a structure have the strong Erd\H{o}s-Hajnal property, it suffices to check for atomic formulas.

We can also reduce checking the $\varepsilon$-strong Erd\H{o}s-Hajnal property for all $\varepsilon$ to a simpler criterion.
\begin{lem}
    A definable predicate $\phi(x_1,\dots,x_n)$ has the strong Erd\H{o}s-Hajnal property if and only if for all $0 \leq r < s \leq 1$,
    there is some $\delta > 0$ such that such that if $A_1,\dots,A_n$ are finite subsets of $M^{x_1},\dots,M^{x_n}$ respectively, then there are $B_1,\dots,B_n$
    with $B_i \subseteq A_i$ and $|B_i| \geq \delta |A_i|$ such that either for all $b \in B_1 \times \dots \times B_n$,
    $\phi(b) < s$, or for all $b \in B_1 \times \dots \times B_n$, $\phi(b) > r$.
\end{lem}
\begin{proof}
    Suppose $\phi$ has the strong Erd\H{o}s-Hajnal property, fix $0 \leq r < s \leq 1$, and let $0 < \varepsilon < s - r$.
    Then we can find $B_1,\dots,B_n$ of adequate size that are $\varepsilon$-homogeneous, implying that either $\phi(b) > r$ or $\phi(b) < s$ is true for all $b \in B_1 \times \dots \times B_n$.
    
    Conversely, assume this new condition holds.
    We will prove for each $n$ that $\phi$ has the $\frac{1}{n}$-strong Erd\H{o}s-Hajnal property.
    By taking a finite minimum, we can find $\delta > 0$ such that for all $0 \leq i < n$,
    given $A_1,\dots,A_n$, there are $B_i \subseteq A_i$ and $|B_i| \geq \delta |A_i|$ such that either for all $b \in B_1 \times \dots \times B_n$,
    $\phi(b) < \frac{i + 1}{n}$, or for all $b \in B_1 \times \dots \times B_n$, $\phi(b) > \frac{i}{n}$.
    Then by a recursive application of this property for each $r = \frac{i}{n}, \frac{i + 1}{n}$,
    we can find $B_i \subseteq A_i$ with $|B_i| \subseteq \delta^n |A_i|$ that satisfy this property for each $(r,s)$ simultaneously.
    Thus there must be some $i$ such that $b \in B_1 \times \dots \times B_n$, $\frac{i}{n} \leq \phi(b) \leq \frac{i + 1}{n}$,
    so $B_1 \times \dots \times B_n$ is $\left(\phi,\frac{1}{n}\right)$-homogeneous.
    It would suffice to reduce the size of the sets only $\log n$ times by a binary search method, improving the constants if necessary.
\end{proof}

Before trying to determine which metric structures have the strong Erd\H{o}s-Hajnal property for all definable predicates, it makes sense to ask whether the metric has this property.
This is true for ultrametrics.

\begin{lem}
    Let $(X,d)$ be a bounded ultrametric space. The metric $d(x,y)$ has the strong Erd\H{o}s-Hajnal property.
\end{lem}
\begin{proof}
    Fix $0 \leq r < 1$, and let $A, B \subseteq X$ be finite.
    We will show that there are $A_0 \subseteq A, B_0 \subseteq B$ with $|A_0| \leq \frac{1}{3}|A|$ and $|B_0| \leq \frac{1}{3}|B|$
    such that either for all $(a,b) \in A_0 \times B_0$, $d(a,b) \leq r$, or for all $(a,b) \in A_0 \times B_0$, $d(a,b) > r$.

    By the ultrametric criterion, $A \cup B$ can be covered with disjoint closed $r$-balls.
    Thus let $A = A_1 \cup \dots \cup A_n$ and $B = B_1 \cup \dots \cup B_n$, where $A_1 \cup B_1$ is contained in a closed $r$-ball,
    but for $i \neq j$, if $u \in A_i \cup B_i$ and $v \in A_j \cup B_j$, then $d(u,v) > r$.
    If there is some $i$ with $|A_i| \geq \frac{1}{3}|A|$ and $|B_i| \geq \frac{1}{3}|B|$, then we can let $A_0 = A_i$ and $B_0 = B_i$.
    Let $S \subseteq \{1,\dots,n\}$ be the set of all $i$ such that $\frac{|A_i|}{|A|} \geq \frac{|B_i|}{|B|}$.
    We can see that $\sum_{i \in S} \frac{|A_i|}{|A|} \geq \sum_{i \in S} \frac{|B_i|}{|B|} = 1 - \sum_{i \not \in S} \frac{|B_i|}{|B|}$,
    from which we can deduce that either $\sum_{i \in S} \frac{|A_i|}{|A|} \geq \frac{1}{2}$ or $\sum_{i \not \in S} \frac{|B_i|}{|B|} \geq \frac{1}{2}$.
    Without loss of generality, assume the former.
    In this case, choose a minimal set $S' \subseteq S$ with $\sum_{i \in S'} \frac{|A_i|}{|A|} \geq \frac{1}{3}$.
    By minimality, for any one $i' \in S'$, $\sum_{i \in S', i \neq i'} \frac{|A_i|}{|A|} < \frac{1}{3}$,
    and thus $\sum_{i \in S', i \neq i'} \frac{|B_i|}{|B|} \leq \frac{1}{3}$.
    By assumption, either $|A_{i'}| < \frac{1}{3}|A|$ and $|B_{i'}| < \frac{1}{3}|B|$.
    As $i' \in S$, meaning $\frac{|A_{i'}|}{|A|} \geq  \frac{|B_{i'}|}{|B|}$,
    we can deduce that $\frac{|B_{i'}|}{|B|} < \frac{1}{3}$,
    so $\sum_{i \in S'} \frac{|B_i|}{|B|} \leq \frac{2}{3}$,
    and $\sum_{i \not\in S'} \frac{|B_i|}{|B|} \geq \frac{1}{3}$.
    Thus we can let $A_0 = \cup_{i \in S'} A_i$ and let $B_0 = \cup_{i \not \in S'} B_i$,
    and get $|A_0| \geq \frac{1}{3}|A|$ and $|B_0| \geq \frac{1}{3}|B|$.
    If $a \in A_0$ and $b \in B$, then there are $i \in S'$ and $j \not \in S'$ with $a \in A_i$ and $b \in B_j$,
    so as $i \neq j$, $d(a,b) > r$.
\end{proof}

\section{Valued Fields}\label{sec_vf}

In \cite{mvf}, Ben Yaacov set up a framework for studying fields with $(\R_{\geq 0},*)$-valued valuations as metric structures.
More specifically, the metric structures are projective spaces over such fields.
\begin{defn}
    Given a field $K$, let $K\mathbb{P}^n$ denote the $n$-dimensional projective space over $K$, whose elements we write in homogeneous coordinates as $[x_0:x_1: \dots : x_n]$,
    which we will generally assume satisfy $\max_i |x_i| = 1$.

    Let $\mathcal{L}_{\mathbb{P}^1}$ be the language considering of the constant symbol $\infty$ and, for each $n \in \N$ and each polynomial $P(x_1,\dots,x_n) \in \Z[x_1,x_1,\dots,x_n]$,
    a relation symbol $||P(\bar x)||$. 

    Given a field $K$ with a multiplicative valuation $|\cdot |$ taking values in $\R_{\geq 0}$,
    we interpret $K\mathbb{P}^1$ as an $\mathcal{L}_{\mathbb{P}^1}$-structure as follows, using homogeneous coordinates:
    \begin{align*}
        d([a:a^*],[b:b^*]) &= |ab^* - a^*b|\\
        \infty &= [1:0]\\
        ||P([a_1:a_1^*],\dots, [a_n:a_n^*])|| &= |P^h(a_1,\dots,a_n,a_1^*,\dots,a_n^*)|,
    \end{align*}
    where $P^h$ is the \emph{homogenization} of $P$.
\end{defn}

\begin{fact}[{\cite[Theorem 1.8]{mvf}}]
    There is a theory $MVF$ in the language $\mathcal{L}_{\mathbb{P}^1}$,
    whose models are (up to isomorphism) exactly the projective lines of valued fields with complete valuation.
\end{fact}

We also can consider a language with more sorts, to encompass all projective spaces over $K$ in one structure:
\begin{defn}
    Let $\mathcal{L}_{\mathbb{P}}$ be the language with sorts $(P^n : n \in \N)$ with the following symbols:
    \begin{itemize}
        \item For each $m,n$, a function $\tensor: \mathbb{P}^m \times\mathbb{P}^n \to \mathbb{P}^{n + m + nm}$
        \item For each $A \in SL_{n+1}(\Z)$, a function $A : \mathbb{P}^n \to \mathbb{P}^n$
        \item For each $n$, a predicate symbol $||\cdot||$ on $\mathbb{P}^n$.
    \end{itemize}

    Given any field $K$ with a multiplicative valuation $|\cdot |$ taking values in $\R_{\geq 0}$,
    we construct an $\mathcal{L}_{\mathbb{P}}$-structure $K\mathbb{P}$ by interpreting $\mathbb{P}^n$ as $K\mathbb{P}^n$.
    We interpret the $\tensor$ symbols as Segre embeddings, interpret the special linear transformation symbols with their natural action on
    $K^{n+1}$, each of which respects the quotient relation that defines $K\mathbb{P}^n$.
    We can then define the other symbols by
    \begin{align*}
        ||[a_0:\dots : a_n]|| &= |a_0|\\
        d(a,b)&=\max_{i < j}|a_ib_j - a_jb_i|.
    \end{align*}
\end{defn}

\begin{fact}
    The $\mathcal{L}_{\mathbb{P}}$-structure $K\mathbb{P}$ and $\mathcal{L}_{\mathbb{P}^1}$-structure $K\mathbb{P}^1$ induced by a valued field $K$ are biinterpretable. 
\end{fact}

The theory $MVF$ admits a natural algebraically closed completion:
\begin{fact}[{\cite[Lemma 2.2]{mvf}}]
    There is a $\mathcal{L}_{\mathbb{P}^1}$-theory $ACMVF$ whose models are precisely the projective lines over algebraically closed fields
    with nontrivial complete valuations.
\end{fact}

To define the theory of real closed metric valued fields, we extend the language:
\begin{defn}
    Extend $\mathcal{L}_{\mathbb{P}^1}$ to the language $\mathcal{L}_{o\mathbb{P}^1}$ by adding a for each such polynomial $P(x_1,\dots,x_n)$ an extra symbol $\langle P(\bar x)\rangle$,
    which we interpret as
    $$\langle P(\bar x)\rangle = d(P(\bar x),\textrm{Sq}),$$
    where Sq is the (closed in any metric valued field) set of squares.
\end{defn}
In a real closed ordered field, Sq is also the set of nonnegative elements, so we can naturally think of this as encoding a linear ordering.
This gives rise to the languages $RCMVF$ and $ORCMVF$ of (ordered) real closed metric valued fields:
\begin{fact}[{\cite[Proposition 3.6, Theorem 3.11]{mvf}}]
    There are a $\mathcal{L}_{\mathbb{P}^1}$-theory $RCMVF$ and a $\mathcal{L}_{o\mathbb{P}^1}$-theory $ORCMVF$
    such that the models of $RCMVF$ are exactly the projective lines of real closed fields with complete non-trivial valuations,
    and models of $ORCMVF$ are exactly the projective lines of such fields where the extra predicate is the distance predicate to the set of nonnegative elements.

    Furthermore, any model of $RCMVF$ admits a unique expansion modelling $ORCMVF$.
    In this expansion, the extra predicate is the distance predicate to the set of nonnegative elements.
\end{fact}

We can now show that these theories are distal.
\begin{thm}
RCMVF is distal.
\end{thm}
\begin{proof}
By \cite[Theorem 5.22]{anderson1}, it suffices to check that if $(a_i : i \in \Q) + b + (c_j : j \in \Q)$ is an indiscernible sequence, with $(a_i: i \in \Q) + (c_j : j \in \Q)$ indiscernible over a singleton $d$,
then $(a_i : i \in \Q) + b + (c_j : j \in \Q)$ is indiscernible over $d$ also.

Let $i_0<\dots <i_{n-1} \in \Q$ and $i_{n+1} < \dots <i_{2n} \in \Q$. We will show that for all $i_n> i_{n-1}$, and all $\varphi(x;y_0,\dots,y_{2n})$, $\varphi(d;a_{i_0},\dots,a_{i_n},c_{i_{n+1}},\dots,c_{i_{2n}}) = \varphi(d;a_{i_0},\dots,a_{i_{n-1}}, b,c_{i_{n+1}},\dots,c_{i_{2n}})$.

By quantifier elimination, it suffices to show that if $\varphi(x;y_0,\dots,y_{2n})$ is an atomic $\mathcal{L}_{o\mathbf{P}^1}$-formula of either the form $\norm{P(x;\bar y)}$ or $\langle P(x;\bar y)\rangle$, then $\vDash \varphi(d;\bar a) = \varphi(d;\bar a')$, whenever $\bar a$ and $\bar a'$ are increasing sequences of length $2n+1$ in $(a_i : i \in \Q) + b + (c_j : j \in \Q)$.
As in the proof of \cite[Theorem 3.12]{mvf}, we find that $\varphi(x;\bar y)$ is a continuous combination of things of the form $\abs{x - f(\bar y)}$ and $\langle x - f(\bar y)\rangle$, where $f$ is a partial $\emptyset$-definable function.
Thus it will suffice to show the desired result for $\varphi$ of those forms.
Given $y$, let $f_0(y) = f(a_{i_0},\dots,a_{i_{n-1}}, y,c_{i_{n+1}},\dots,c_{i_{2n}})$.
We wish to show that $\abs{d - f_0(y)}$ and $\langle d - f_0(y)\rangle$ are constant on the indiscernible sequence $I = (a_i: i > i_{n-1}) + b + (c_i : i < i_{n+1}$.
The sequence $f_0(y): y \in I$ will itself be indiscernible, and thus monotone, and $f_0(y): y \in I \setminus \{b\}$ is indiscernible over $d$, so $\abs{d - f_0(y)}$ and $\langle d - f_0(y)\rangle$ are constant over $I \setminus \{b\}$.As for any values $r,s$, the set of $y$ such that $\abs{d - f_0(y)} = r$ and $\langle d - f_0(y)\rangle = s$ is order-convex, we see that $\abs{d - f_0(y)}$ and $\langle d - f_0(y)\rangle$ must also be constant on all of $I$ as desired.
\end{proof}



It is also possible to interpret ACMVF$_{(0,0)}$ in RCMVF, and thus show the (not definable) strong Erd\H{o}s-Hajnal property for that stable theory.
In general, if $K$ is a metric valued field, it is complete and thus Henselian, so if $L / K$ is a finite-degree field extension, and thus $L$ is a metric valued field with the unique valuation extending the valuation on $K$.
We claim that $L$ is interpretable in $K$, and as a consequence, ACMVF$_{(0,0)}$ is interpretable in RCMVF.

\begin{thm}
    Let $K$ be a metric valued field, and let $L$ be a finite extension of $K$, with the unique valuation extending that of $K$. Then $L$ is interpretable in $K$.
\end{thm}
\begin{proof}
    Let $d = [L : K]$, and let $\alpha \in L$ be the root of a monic degree-$d$ polynomial in $K[X]$ such that $L = K(\alpha)$.
    
    Roughly speaking, we will represent an element $W \in L\mathbb{P}^1$ with two elements of $K\mathbb{P}^d$, spelling out $W$ and $W^{-1}$ in the $\alpha$-basis.
    From the construction of $K\mathbb{P}$, we see that for any homogeneous polynomial $P \in \Z[X_0,\dots,X_d]$,
    $|P(X_0,\dots,X_d)|$, evaluated at a representative where $\bigcap_{i = 0}^d |X_i| = 1$, is a definable predicate on $K\mathbb{P}^d$ without parameters.
    If instead $P \in K[X_0,\dots,X_d]$, this will be definable with parameters.

    Our interpretation will use the element $[X_0:\dots:X_d] \in K\mathbb{P}^n$ to represent $[\sum_{i = 0} X_i \alpha^i : X_d] \in L\mathbb{P}^1$.
    This is well-defined, and it is surjective because any $[Y : 1]$ can be represented by some $[X_0:\dots:X_{d-1} : 1]$, and the single point at infinity $[1:0]$ can be represented by $[1:0:\dots:0]$.
    We wish to show that for any polynomial $P \in \Z[Y_1,\dots,Y_n,Y_1^*,\dots,Y_n^*]$, the predicate $||P||$, evaluated at $[X_{01}:\dots:X_{d1}],\dots,[X_{0n}:\dots:X_{dn}]$, is a definable predicate.
    To do this, we will first show that for any polynomial $Q \in \Z[Y_1,\dots,Y_n,Z_1,\dots,Z_n]$, homogeneous in each pair $(Y_i,Z_i)$, the function
    $$\left|Q\left(\sum_{i = 0}X_{i1} \alpha^i,\dots,\sum_{i = 0}X_{in} \alpha^i,X_{d1},\dots,X_{dn}\right)\right|$$
    is a definable predicate.
    Then if $P^h$ is the homogenization of $P$, we can evaluate $||P||$ by calculating
    $$\left|P^h\left(\sum_{i = 0}X_{i1} \alpha^i,\dots,\sum_{i = 0}X_{in} \alpha^i,X_{d1},\dots,X_{dn}\right)\right|$$
    and then correcting for the max norm $|\sum_{i = 0}X_{ij} \alpha^i| \vee |X_{dj}|$ for each $j$, by dividing by the appropriate power of $|\sum_{i = 0}X_{ij} \alpha^i| \vee |X_{dj}|$,
    which is itself a nowhere-zero definable predicate, as it is the maximum of the valuations of two homogeneous polynomials, namely $|Y_j|$ and $|Y_j^*|$.

    For all $x \in L$, we can understand the valuation $|x|$ in terms of the norm $|N_{L/K}(x)| = |x|^d$,
    as $|N_{L/K}(x)| = |\prod_{i = 1}^d x_i| = |x|^d$, where $\{x_1,\dots,x_d\}$ are the conjugates of $x$ under the $d$ automorphisms of $L / K$, each of which has 
    $|x_i| = |x|$ by Henselianity of the complete field $K$.
    The norm $N_{L/K}\left(\sum_{i = 0}^{d-1} X_i \alpha^i\right)$ can be defined as a determinant, and in particular is a homogeneous degree $d$ polynomial in $K[X_0,\dots,X_{d-1}]$,
    so if $Q_0,\dots,Q_{d-1} \in K[X_{ij} : 0 \leq i \leq d, 1 \leq j \leq n]$ are polynomials homogeneous in each tuple $(X_{0i},\dots,X_{di})$ of the same multidegree (or zero),
    then $N_{L/K}\left(\sum_{i = 0}^{d-1} Q_i \alpha^i\right)$ is itself a homogeneous polynomial in $K[X_{ij} : 0 \leq i \leq d, 1 \leq j \leq n]$,
    so $\left|\sum_{i = 0}^{d-1} Q_i \alpha^i\right| = \left|N_{L/K}\left(\sum_{i = 0}^{d-1} Q_i \alpha^i\right)\right|^{1/d}$ will be a definable predicate.
    For each $Q \in \Z[Y_1,\dots,Y_n,Z_1,\dots,Z_n]$ is homogeneous in each pair $(Y_i,Z_i)$, with $d_i = \deg_{Y_i}(Q) + \deg_{Z_i}(Q)$, then we can express
    $$Q\left(\sum_{i = 0}X_{i1} \alpha^i,\dots,\sum_{i = 0}X_{in} \alpha^i,X_{d1},\dots,X_{dn}\right) = \sum_{i = 0}^{d-1} Q_i\alpha^i,$$
    where each $Q_i$ is homogeneous in each tuple $(X_{0i},\dots,X_{di})$ with the same multidegree $d_i = \sum_{j = 0}^d \deg_{X_{ji}}Q$, unless it is zero.
    Thus $|Q| = \left|\sum_{i = 0}^{d-1} Q_i \alpha^i\right|$ is definable.
\end{proof}

\section{Dual Linear Continua}\label{sec_dlc}
In \cite{no_trans}, Ben Yaacov analyzes an $\aleph_0$-categorical metric structure whose homeomorphism group is $\mathrm{Hom}^+([0,1])$, the group of increasing homeomorphisms of $[0,1]$ under the topology of uniform convergence.
We call models of the theory of this structure \emph{dual linear continua}, because we will show that they are in correspondence with linear continua with endpoints, which are characterized by the following definition and fact.

\begin{defn}
A \emph{linear continuum} is a dense linear ordering with the least upper bound property.
\end{defn}

\begin{fact}\label{fact_lin_cont}
A linear order is connected in the order topology if and only if it is a linear continuum, and it is connected and compact if and only if it is a linear continuum with endpoints.
\end{fact}

\begin{defn}
    Given a linear ordering $L$, let $M_L$ be the set of functions $f : L \to [0,1]$ such that
    \begin{itemize}
        \item $f$ is nondecreasing,
        \item $f$ is continuous with respect to the order topology on $L$,
        \item $\inf_{x} f(x) = 0$,
        \item $\sup_{x} f(x) = 1$.
    \end{itemize}

    We give $M_L$ the sup metric.
\end{defn}

In \cite{no_trans}, $M_{[0,1]}$ is given additional structure, which makes its automorphism group $\mathrm{Hom}^+([0,1])$.
If $f \in \mathrm{Hom}^+([0,1])$, then $f$ acts on $M_{[0,1]}$ by composition, sending $g \in M_{[0,1]}$ to $g \circ f^{-1} \in M_{[0,1]}$.
In order to describe the analogous structure on $M_L$ for other linear orders, we will first describe the type spaces of this structure.

\begin{lem}
    The type $\mathrm{tp}(f_1,\dots,f_n)$ is determined exactly by the image of the function $(f_1,\dots,f_n) : [0,1] \to [0,1]^n$.
\end{lem}
\begin{proof}
    In \cite{no_trans}, it is shown that the type of $(f_1,\dots,f_n)$ is determined by the function $(g_1,\dots,g_n)$ such that
    $\frac{1}{n}\sum_{i = 1}^n f_i = \id$ and 
    $(g_1,\dots,g_n) \circ \left(\frac{1}{n}\sum_{i = 1}^n f_i\right) = (f_1,\dots,f_n)$.
    This correspondence is a homomorphism between the space of such function tuples and the space of types.
    The function $(g_1,\dots,g_n)$ has the same image as $(f_1,\dots,f_n)$, so the type of a tuple determines its image.
    
    We now consider two tuples $(f_1,\dots,f_n)$ and $(g_1,\dots,g_n)$ with the same image, and show that they have the same type.
    We may assume that $\frac{1}{n}\sum_{i = 1}^n f_i = \frac{1}{n}\sum_{i = 1}^n g_i = \id$, and show that the tuples are equal.
    For any $t \in [0,1]$, there is some $t' \in [0,1]$ such that $(f_1(t),\dots,f_n(t)) = (g_1(t'),\dots,g_n(t'))$.
    However, $\sum_{i = 1}^n f_i(t) = nt$ and $\sum_{i = 1}^n f_i(t') = nt'$, so $t = t'$, and the tuples are equal.

\end{proof}

Given that correspondence, if $p \in S_n(\emptyset)$ is a type in this theory, let $\mathrm{im}(p)$ be the image of any realization of $p$ in $(M_{[0,1]})^n$.
(Such a realization exists because $M_{[0,1]}$ is $\aleph_0$-categorical and thus $\aleph_0$-saturated.)
We will use this characterization to understand the topology and metric on the type space, but first, some simple topological lemmas.
(Recall that while in general, the metric on a type space does not induce the topology, it does in the $\aleph_0$-categorical case.)

\begin{defn}
    If $I$ is a set, give $[0,1]^I$ the product order defined by $(x_i : i \in I) \leq (y_i : i \in I)$.
    A \emph{chain from 0 to 1} in $[0,1]^I$ is a set $C \subseteq [0,1]^I$ which is a chain in the product order and contains the constant tuples with values $0$ and $1$.
\end{defn}

\begin{lem}
Let $I$ be a set, and let $C \subseteq [0,1]^I$ be a chain from 0 to 1. Then the subset topology on $C$ is the order topology, and $C$ is compact.
\end{lem}
\begin{proof}
    First, we note that for each $i \in I, r \in [0,1]$, there is some $f \in C$ such that $f(i) = r$.
    If not, then we may partition $C$ with the two disjoint open sets $\{f \in C : f(i) < r\}$ and $\{f \in C : f(i) > r\}$, contradicting connectedness.

    To show the topologies agree, it suffices, without loss of generality, to show that for $f \in C$, the closed interval $[0,f] \subseteq C$ is closed in the subset topology, and that 
    for any $r \in [0,1]$ and $i \in I$, the set $\{f \in C: f(i) \leq r\}$ is closed in the order topology.

    By definition, the closed interval $[0,f]$ is the set $\bigcap_{i \in I} \{g \in C : g(i) \leq f(i)\}$, which is closed in the subset topology.

    Meanwhile, $\{f \in C: f(i) \leq r\} = \bigcap_{g \in C : g(i) > r} [0,g]$.
    For each $g \in C$ with $g(i) > r$, it follows that $\{f \in C: f(i) \leq r\} \subseteq [0,g]$ because $C$ is linearly ordered.
    Also, for each $s \in (r,1]$, there is some $g \in C$ with $g(i) = s$, so $\bigcap_{g \in C : g(i) > r} [0,g] \subseteq \{f \in C: f(i) \leq r\}$.

    In the order topology, by Fact \ref{fact_lin_cont}, connectedness and endpoints imply compactness.
\end{proof}

\begin{lem}\label{lem_connected_chain_path}
Let $I$ be a countable set.
Then if $C \subseteq [0,1]^I$ is a connected chain from 0 to 1, then there are continuous, surjective, nondecreasing functions $(f_i : i \in I)$ such that $\mathrm{im}(f_i : i \in I)=C$.
\end{lem}
\begin{proof}
    By taking a bijection, we may assume that $I$ is an initial segment of $\N$.
    If $I = \{0,\dots,n\}$, we let $C'$ be the set of all $c \in [0,1]^\N$ such that $c\upharpoonright_I \in C$.
    This is clearly also a chain from 0 to 1, which is connected because it is the image of $C$ under a continuous map that just duplicates coordinates.
    If $C'$ is compact, then $C$ is also the image of $C'$ under a continuous map that deletes coordinates, so $C$ is compact.
    If there are continuous, surjective, nondecreasing functions $(f_i : i \in \N)$ such that $\mathrm{im}(f_i : i \in I)=C'$, then $(f_i : i \leq n)$ will suffice for $C$, so we may assume that $I = \N$.

    Define $g : C \to [0,1]$ by $g(c) = \sum_{i} c(i) 2^{-i}$.
    This is a strictly increasing continuous function, which attains values 0 and 1. Because it is defined on a connected set, its image is connected, so $g$ is surjective.
    Because $C$ is a chain and $g$ is strictly increasing, $g$ is also injective, so it is a homeomorphism as its domain and codomain are compact Hausdorff.
    Thus we can let each $f_i$ be the $i$th coordinate map of $g^{-1}$. These are continuous, surjective, and nondecreasing, and $(f_i : i \in I) = g^{-1}$, whose image is $C$.
\end{proof}

We now characterize the type spaces.
\begin{lem}\label{lem_tp_metric}
    The map $\mathrm{im}$ is an isometry between the type space $S_n(\emptyset)$ in the theory of $M_{[0,1]}$, and the set of all connected chains from 0 to 1 in $[0,1]^n$, given the Hausdorff metric as compact subsets of $[0,1]^n$, itself given the sup metric.
\end{lem}
\begin{proof}
    First we confirm that the image of $\mathrm{im}$ is what we claim, and then we will show that $\mathrm{im}$, as a function to the set of compact subsets of $[0,1]^n$ with the Hausdorff metric,
    is an isometry.
    This is an injective continuous map between compact Hausdorff spaces, so it is a homeomorphism onto its image, and the rest of the lemma will follow from these two claims.

    Clearly if $(f_1,\dots,f_n) \in M^n$, then $\mathrm{im}(f_1,\dots,f_n)$ is a connected chain from $0$ to $1$.
    If $P \subseteq [0,1]^n$ is a connected chain from 0 to 1, then by Lemma \ref{lem_connected_chain_path}, there is some $(f_1,\dots,f_n) : [0,1]\to [0,1]^n$ with $P$ as its image,
    and each $f_i$ continuous, surjective, and nondecreasing.
    Thus $P = \mathrm{im}(f_1,\dots,f_n)$ and $f_1,\dots,f_n \in M_{[0,1]}^n$, so such sets are exactly the images of $n$-types over $M$.

    Now we check that the metric coincides with the metric on types.
    Let $p,q \in S_n(\emptyset)$. First we show that $d(\mathrm{im}(p),\mathrm{im}(q)) \leq d(p,q)$.
    As
    $$d(\mathrm{im}(p),\mathrm{im}(q)) = \max\left(\sup_{x \in \mathrm{p}}d(x,\mathrm{im}(q)),\sup_{y \in \mathrm{im}(q)} d(y,\mathrm{im}(p))\right)$$
    and $d(p,q) = \inf_{\bar f, \bar g : \mathrm{tp}(\bar f) = p, \mathrm{tp}(\bar g) = q} d(\bar f,\bar g)$,
    it suffices to show, without loss of generality, that for each $\bar f, \bar g$ such that $\mathrm{tp}(\bar f) = p, \mathrm{tp}(\bar g) = q$, and each
    $x \in \mathrm{im}(p)$, $d(x,\mathrm{im}(q)) \leq d(\bar f, \bar g)$.
    Let $t$ be such that $x = (\bar f)(t)$. Then 
    $$d(x,\mathrm{im}(q)) \leq d(\bar f(t),\bar g(t)) \leq d(\bar f,\bar g).$$
    
    It now suffices to show that there exist $\bar f,\bar g$ with $\mathrm{tp}(\bar f) = p, \mathrm{tp}(\bar g) = q$ such that
    $d(\bar f, \bar g) \leq d(\mathrm{im}(p),\mathrm{im}(q))$.
    Let $\bar f^*, \bar g^*$ be such that $\mathrm{tp}(\bar f^*) = p, \mathrm{tp}(\bar g^*) = q$, and $\frac{1}{n}\sum_{i = 1}^n f_i^* = \frac{1}{n}\sum_{i = 1}^n g_i^* = \id$.
    We will show that there exists a connected chain $C \subseteq [0,1]^2$ containing $(0,0)$ and $(1,1)$, such that for all $(t,t') \in C$,
    $d(\bar f^*(t),\bar g^*(t')) \leq d(\mathrm{im}(p),\mathrm{im}(q))$.
    By Lemma \ref{lem_connected_chain_path}, there are continuous, surjective, nondecreasing functions $f',g'$ such that $\mathrm{im}((f',g'))=C$.
    Then let $\bar f = \bar f^* \circ f'$ and $\bar g = \bar g^* \circ g'$.
    We know that $\mathrm{tp}(\bar f) = \mathrm{tp}(\bar f^*) = p$ and $\mathrm{tp}(\bar g) = \mathrm{tp}(\bar g^*) = q$,
    and we know that for each $t$, $(f'(t),g'(t)) \in C$, so $d(\bar f(t),\bar g(t)) \leq d(\mathrm{im}(p),\mathrm{im}(q))$, so these $\bar f,\bar g$ will suffice.

    To construct the chain $C$, first assume without loss of generality that $d(\mathrm{im}(p),\mathrm{im}(q)) = \sup_{x \in \mathrm{im}(p)} d(x,\mathrm{im}(q))$.
    Then for all $t \in [0,1]$, let $Y_t = \{t' : d(f^*(t),g^*(t')) \leq d(\mathrm{im}(p),\mathrm{im}(q))\}$.
    This is a closed interval in $\mathrm{im}(q)$. It will always be nonempty by the assumption that $d(\mathrm{im}(p),\mathrm{im}(q)) = \sup_{x \in \mathrm{im}(p)} d(x,\mathrm{im}(q))$.
    For each $t$, let $y_t = \min Y_t$.
    Then $t \mapsto y_t$ is a nondecreasing function from $[0,1] \to [0,1]$, so it is piecewise continuous with countably many discontinuities.
    Thus filling in these discontinuities with countably many vertical intervals turns the graph of this function into a path from $(0,0)$ to $(1,1)$, which we call $C$.
    It now suffices to show that for all $(t,t') \in C$, $d(\bar f^*(t),\bar g^*(t')) \leq d(\mathrm{im}(p),\mathrm{im}(q))$, that is, $t' \in Y_t$.
    If $t' = y_t$, this follows by definition, so we may assume that $(t,t')$ lies on one of the vertical segments, so $\lim_{s \to t^-} y_s \leq t' \leq \lim_{s \to t^+} y_s$.
    Because $\{(t,t'): t' \in Y_t\}$ is closed, we find that $\lim_{s \to t^-} (s,y_s)$ and $\lim_{s \to t^+} (s,y_s)$ are both points of $\{(t,t'): t' \in Y_t\}$.
    Thus $t'$ lies between two points in $Y_t$, which is an interval, so $t' \in Y_t$.
\end{proof}

Now we can define the structure on $M_L$ for any linear continuum with endpoints $L$, by defining the type of any tuple in the type spaces $S_n(\emptyset)$ of the structure $M_{[0,1]}$,
which from this point on we simply call $S_n$.
\begin{defn}
    Given a linear continuum with endpoints $L$ and $f_1,\dots,f_n \in M_L$, let $\mathrm{tp}(f_1,\dots,f_n)$ be the unique type $p$ with $\mathrm{im}(p) = \mathrm{im}(f_1,\dots,f_n)$. 
\end{defn}
For this to genuinely define a structure on $M_L$, it suffices to check that for each $C$-Lipschitz $n$-ary definable relation, which is interpreted as some $h : S_n \to [0,1]$,
that $h \circ \mathrm{tp} : M_L^n \to [0,1]$ is also $C$-Lipschitz. This follows from $\mathrm{tp}$ being a contraction.
By the proof of Lemma \ref{lem_tp_metric}, for any $p,q \in S_n$,
$d(p,q) = \inf_{\bar f,\bar g \in M_L : \mathrm{tp}(\bar f) = p, \mathrm{tp}(\bar g) = q}d(\mathrm{tp}(\bar f),\mathrm{tp}(\bar g)),$
so this is indeed a contraction.

As the finite-dimensional type spaces coincide with those of $M_{[0,1]}$, these structures are all elementarily equivalent.
We now characterize arbitrary type spaces over models of this theory.

\begin{lem}\label{lem_inf_connected_chain}
Let $x$ be a possibly infinite variable tuple.
The type space $S_x$ consists of all connected chains from 0 to 1 in $[0,1]^x$. 
\end{lem}
\begin{proof}
    The type space $S_x$ is the topological inverse limit of all $S_y$ where $y$ is a finite subtuple of $x$.
    Thinking of each $S_y$ as the set of connected chains from 0 to 1 in $[0,1]^y$, the restriction maps $S_y \to S_z$ for $z \subseteq y$ are given by restricting variables of chains.
    This means that $S_x$ is homeomorphic to the inverse image of these spaces $S_y$ as a subset of $[0,1]^x$, and it suffices to determine which sets in $[0,1]^x$ have connected chains as each finite projection.
    Such sets are exactly inverse limits of directed systems of connected chains from 0 to 1 in finite-dimensional spaces - that is, those sets whose projections to finite-dimensional spaces are connected chains from 0 to 1.
    It is clear that such projections are chains from 0 to 1 if and only if the original set is a chain from 0 to 1, and that the projections of a connected set are all connected.
    It suffices now to show that a set $X$ whose finite-dimensional projections are connected chains from 0 to 1 is connected.
    Such a set $X$ is an inverse limit of closed sets by Lemma \ref{lem_connected_chain_path}, so it is itself closed, and thus compact.
    Thus if $X$ is disconnected, there are basis open sets $A, B$ such that $X \cap A, X \cap B$ partition $X$.
    However, using the standard basis of the topology, this means there is a finite subtuple $y$ of $x$ and open sets $A, B \subseteq [0,1]^y$ such that $A, B$ partition $\pi(X)$, where $\pi$ projects $[0,1]^x$ onto $[0,1]^y$.
    This contradicts the connectedness of $\pi(X)$.
\end{proof}

We can now fully characterize models of $\mathrm{Th}(M_{[0,1]})$.

\begin{thm}\label{thm_lin_cont}
    If $M \equiv M_{[0,1]}$, then $M$ is isomorphic to $M_L$ for some linear continuum with endpoints $L$.
\end{thm}
\begin{proof}
    Let $p \in S_M$ be the type of $M$ enumerated as a tuple, and let $L = \mathrm{im}(p)$.
    By Lemma \ref{lem_inf_connected_chain}, this is a connected chain from 0 to 1, and is thus a linear continuum with endpoints.
    We now define $f : M \to M_L$. If $m \in M, x = (x_m : m \in M) \in L$, then we define $f(m)(x)=x_m$.

    Let $m_1,\dots,m_n \in M$. We wish to show that $\mathrm{tp}(m_1,\dots,m_n) = \mathrm{tp}(f(m_1),\dots,f(m_n))$, by showing the images of the types are equal.
    We know that $\mathrm{im}(\mathrm{tp}(m_1,\dots,m_n))$ is just the projection of $L = \mathrm{im}(p)$ onto the coordinates $(m_1,\dots,m_n)$,
    coinciding precisely with $\mathrm{im}(\mathrm{tp}(f(m_1),\dots,f(m_n))) = \mathrm{im}(f(m_1),\dots,f(m_n))$.
    As $f$ preserves types, it is also an isometry.
    
    Let $g \in M_L$, fix $n \in \N$, and then let $x_1< \dots < x_n \in L$ be such that for each $i$, $g(x_i) = \frac{i}{n + 1}$ for each $1 \leq i \leq n$.
    Then for each $1 \leq i < n$, there is some $m_i$ with $f(m_i)(x_i) < f(m_i)(x_{i + 1})$.
    As the sequence $\frac{1}{n - 1}\sum_{i = 1}^{n-1}f(m_i)(x_j)$ for $1 \leq j \leq n$ is strictly increasing, there is a continuous monotone bijection $\theta : [0,1] \to [0,1]$
    such that for each $1 \leq j < n$, $\theta\left(\frac{1}{n - 1}\sum_{i = 1}^{n-1}f(m_i)(x_j)\right) = \frac{j}{n + 1}$, and also $\theta(0) = 0$ and $\theta(1) = 1$.
    
    We claim that the function that sends $g_1,\dots,g_{n-1} \in M_{[0,1]}$ to the function given by $t \mapsto \theta\left(\frac{1}{n - 1}\sum_{i = 1}^{n-1}g_i(t)\right)$ is definable.
    To show this, we first show that the predicate taking the average of $n - 1$ elements of $M_{[0,1]}$ is definable, and then check that composition with $\theta$ is definable.
    As shown in \cite{defOSgrp}, it suffices to show that these functions are type-definable.
    The graph of the average function, viewed as a subset of the type space, consists of all connected chains from 0 to 1 through $[0,1]^{n}$ contained in the closed subset
    $x_n = \frac{1}{n-1}\sum_{i = 1}^{n-1}x_i$.
    The set of connected chains from 0 to 1 contained in a closed subset is closed, as the type space topology on chains is given by the Vietoris topology.
    Thus the graph of the average function is closed, as is the graph of composition with $\theta$, consisting of all chains residing in the (closed) graph of $\theta$.

    By the definability of this function, as $M \equiv M_{[0,1]}$, there is some $m \in M$ such that for each $(a_1,\dots,a_n;b) \in (m_1,\dots,m_n;m)$, $b = \theta\left(\frac{1}{n - 1}\sum_{i = 1}^{n-1}a_i\right)$.
    Thus for each $1 \leq j \leq n$, $f(m)(x_j) = \theta\left(\frac{1}{n - 1}\sum_{i = 1}^{n-1}f(m_i)(x_j)\right) = \frac{j}{n + 1} = g(x_j)$.
    Thus for all $x \in L$, $f(m)(x)$ and $g(x)$ lie in a common interval $\left[\frac{j}{n+1},\frac{j+1}{n+1}\right]$, and thus $d(f(m),g)\leq \frac{1}{n+1}$.
    Because $f$ is an isometry and $M$ is complete, this shows us that $f$ is a bijection. As $f$ also preserves types, it is an isomorphism.
\end{proof}

\subsection{Indiscernibles}

The structure $M_{[0,1]}$ is known to be NIP and in a certain precise way, purely unstable.\cite[Corollary 4.17]{ibarlucia1}
We will study its indiscernible sequences, and show that it is distal, in analogy to the $\aleph_0$-categorical structure $(\Q,<)$, whose automorphism group is $\mathrm{Hom}^+(\Q \cap [0,1])$.

We now show that all types in the theory of $M_{[0,1]}$ are determined by types on two variables.
\begin{lem}\label{lem_type_two}
If $p \in S_n$, then $p(x_1,\dots,x_n)$ is implied by $\bigcup_{i < j} p\upharpoonright_{x_i x_j}$, where $p\upharpoonright_{x_i x_j}$ is the restriction of $p$ to the variable tuple $x_ix_j$.

In general, if $p \in S_n$ and $(a_1,\dots,a_n)$ is such that for each $i < j$, $(a_i,a_j) \in \mathrm{im}(p\upharpoonright_{x_i x_j})$, then $(a_1,\dots,a_n) \in \mathrm{im}(p)$.
\end{lem}
\begin{proof}
    Let $p \in S_n(\emptyset)$ be a type with $(a_i, a_j) \in \mathrm{im}(p\upharpoonright_{x_i x_j})$ for each $i < j$.
    Let $(f_1,\dots,f_n)$ be a realization of $p$ in $M_{[0,1]}^n$.
    Then consider the $n$ closed intervals $f_i^{-1}(\{a_i\})$.
    Because $(a_i, a_j) \in \mathrm{im}(p\upharpoonright_{x_i x_j})$, the intervals $f_i^{-1}(\{a_i\})$ and $f_j^{-1}(\{a_j\})$ must nontrivially intersect.
    Closed real intervals have the 2-Helly property - any family of intervals that intersect pairwise has a nontrivial intersection, so $(a_1,\dots,a_n)$ must be in the image of $(f_1,\dots,f_n)$.
    
    If $p,q$ are types such that for each $i<j$, $p\upharpoonright_{x_i x_j} = q\upharpoonright_{x_i x_j}$, then for each $(a_1,\dots,a_n) \in \mathrm{im}(p)$,
    we know that for each $i < j$, $(a_i,a_j) \in \mathrm{im}(q\upharpoonright_{x_i x_j})$, so $(a_1,\dots,a_n) \in \mathrm{im}(q)$.
    Thus $p = q$, and these types are determined by their restrictions to two variables.
\end{proof}

We now analyze (possibly finite) indiscernible sequences in structures elementarily equivalent to $M_{[0,1]}$.
Whenever $M \equiv M_{[0,1]}$, by Theorem \ref{thm_lin_cont}, we may assume that $M = M_L$ for some linear continuum $L$.

\begin{lem}\label{lem_mono}
    Let $L$ be a linear continuum with endpoints and let $(f_i : i \in I)$ be an indiscernible sequence in one variable in $M_L$, and let $(a_i : i \in I) \in \mathrm{im}(f_i : i \in I)$.
    Then $(a_i : i \in I)$ is either nondecreasing or nonincreasing.
\end{lem}
\begin{proof}
    It suffices to show that if $(f,g,h) \in M_L^3$ is indiscernible, and there is some $t_0 \in L$ such that $(f,g,h)(t_0) = (a,b,c)$, then $a \leq b \leq c$ or $a \geq b \geq c$.

    If $b = c$, this is trivial, so we may assume without loss of generality that $b < c$.
    By indiscernibility, there also exists $t_1$ such that $(f,g)(t_1) = (b,c)$.
    Then because $g(t_0) = b < c = g(t_1)$, we know that $t_0 < t_1$, and thus $a = f(t_0) \leq f(t_1) = b$.
\end{proof}

\begin{lem}\label{lem_diag}
    Let $L$ be a linear continuum with endpoints and let $(f_i : i \in I)$ be an indiscernible sequence in one variable in $M_L$ of length at least 3, and let $i < j$ be elements of $I$, with 
    $(a,b) \in \mathrm{im}(f_i, f_j)$.
    Then either $(a,a)\in \mathrm{im}(f_i, f_j)$ or $(b,b) \in \mathrm{im}(f_i, f_j)$.
\end{lem}
\begin{proof}
    Let $(f,g,h) \in M_L^3$ be indiscernible, with $(a,b) \in \mathrm{im}(f, g)$. Without loss of generality, assume $a < b$.
    There must be some $c$ with $(a,b,c) \in \mathrm{im}(f,g,h)$, and by Lemma \ref{lem_mono}, $b \leq c$.
    If $b = c$, then $(b,b) \in \mathrm{im}(g,h) = \mathrm{im}(f,g)$, and we are done.
    Otherwise, $b < c$.
    
    Let $t_0,t_1,t_2 \in L$ be such that$(f,h)(t_0) = (a,b)$, $(g,h)(t_1) = (a,c)$, $(f,g,h)(t_2) = (a,b,c)$.
    By monotonicity of $h$, $t_0 < t_1$, and by monotonicity of $g$, $t_1 < t_2$.
    By monotonicity of $f$, then, $f(t_1) = a$, so $(a,a)\in \mathrm{im}(f,g)$.

    As the desired property is true for length-3 indiscernible sequences in $M_L$, it is also true for all longer indiscernibles.
\end{proof}

\begin{lem}\label{lem_1dim_ind}
    Let $L$ be a linear continuum with endpoints. Any indiscernible sequence in one variable in $M_L$ is distal.
\end{lem}
\begin{proof}
    By \cite[Lemma 5.3]{anderson1}, it suffices to show that if $(f_i : i \in I)$ is a sequence of elements in $M_L$, and $i_0 < i_1$ are such that $(i_0,i_1)$ is infinite and removing either $f_{i_0}$ or $f_{i_1}$ makes the sequence indiscernible, then $(f_i : i \in I)$ is indiscernible.
    To do this, we show that $\mathrm{tp}(f_{i_0},f_{i_1}) = \mathrm{tp}(f_{i},f_{j})$ for all other $i < j$.
    This will even apply for sequences of finite length - if $f_1,\dots,f_5 \in M$ are such that $(f_1,f_2,f_3,f_5)$ and $(f_1,f_3,f_4,f_5)$ are indiscernible,
    then $\mathrm{tp}(f_2,f_4) = \mathrm{tp}(f_1,f_3)$, and this will imply indiscernibility for any infinite sequence containing these elements and satisfying the above properties.

    Let $(a,b) \in \mathrm{im}(f_1,f_3)$. By Lemma \ref{lem_diag}, either $(a,a)\in \mathrm{im}(f_1,f_3)$ or $(b,b) \in \mathrm{im}(f_1,f_3)$.
    Without loss of generality, we may assume the former case.
    By Lemma \ref{lem_type_two} and indiscernibility, we see that $(a,a,b) \in \mathrm{im}(f_1,f_3,f_4)$, so there is some $t$ with $(f_1,f_3,f_4)(t) = (a,a,b)$, and by Lemma \ref{lem_mono},
    we have that $f_2(t) = a$ as well, so $(a,b) \in \mathrm{im}(f_2,f_4)$, so $\mathrm{tp}(f_2,f_4) = \mathrm{tp}(f_1,f_3)$ as desired.
\end{proof}

We now show that the interaction between tuples in any $M_L$ can be coded by their averages.
\begin{lem}\label{lem_code}
Let $L$ be a linear continuum with endpoints and let $\bar f = (f_1,\dots,f_n)$ and $\bar g = (g_1,\dots,g_n)$ be tuples in $M_L$.
Define $\hat{f} = \frac{1}{n}\sum_{i = 1}^n f_n$ and $\hat{g} = \frac{1}{n}\sum_{i = 1}^n g_n$.
Then $\mathrm{tp}(\bar f, \bar g)$ is determined by $\mathrm{tp}(\bar f), \mathrm{tp}(\bar g), \mathrm{tp}(\hat{f},\hat{g})$.
\end{lem}
\begin{proof}
    Clearly $\mathrm{im}(\bar f,\hat{f})$ is determined by $\mathrm{im}(\bar f)$.
    By the monotonicity of $\bar f$ and the surjectivity of $\hat{f}$, for any $a \in L$, there is exactly one $\bar a \in L^n$ such that $(\bar a, a) \in \mathrm{im}(\bar f,\hat{f})$.
    Thus $\mathrm{im}(\bar f, \bar g)$ consists of all $(\bar a,\bar b)$ such that if $\hat{a}, \hat{b}$ are the averages of $\bar a,\bar b$,
    then $\bar a \in \mathrm{im}(\bar f), \bar b \in \mathrm{im}(\bar g),$ and $(\hat{a},\hat{b}) \in \mathrm{im}(\hat{f},\hat{g})$.
\end{proof}

\begin{prop}
Then the structure $M_{[0,1]}$ is distal.
\end{prop}
\begin{proof}
    Let $(f_i : i \in I)$ be an indiscernible sequence in an elementary extension of $M_{[0,1]}$, which we may assume is $M_L$ for some linear continuum $L$ with endpoints.
    As in the proof of Lemma \ref{lem_1dim_ind} but with longer tuples,
    we will just show that if $(f_1,\dots,f_5) \in (M_{L}^n)^5$ is such that $(f_1,f_2,f_3,f_5)$ and $(f_1,f_3,f_4,f_5)$ are indiscernible, then
    the $\mathrm{tp}(f_2,f_4) = \mathrm{tp}(f_1,f_3)$.

    Clearly for each $1 \leq i,j \leq 5$, $\mathrm{tp}(f_i) = \mathrm{tp}(f_j)$.
    For $1 \leq i \leq 5$, let $\hat{f_i} \in M_L$ be the pointwise average of the tuple $f_i$.
    By indiscernibility of these subsequences, we can deduce that $(\hat{f_1},\dots,\hat{f_5})$ is indiscernible, so by the proof of Lemma \ref{lem_1dim_ind},
    $\mathrm{tp}(\hat{f_2},\hat{f_4}) = \mathrm{tp}(\hat{f_1},\hat{f_3})$.
    By Lemma \ref{lem_code}, this constrains the types of the tuples enough that $\mathrm{tp}(f_2,f_4) = \mathrm{tp}(f_1,f_3)$.
\end{proof}

We can say more about indiscernibles.
\begin{thm}
    If $p \in S_2(\emptyset)$, then $p$ is the type $(f_i,f_j)$ with $i < j$ in some infinite indiscernible sequence $(f_i : i \in I)$
    if and only if for all $(a,b) \in \mathrm{im}(p)$, either $(a,a) \in \mathrm{im}(p)$ or $(b,b) \in \mathrm{im}(p)$.
\end{thm}
\begin{proof}
    Lemma \ref{lem_diag} tells us that the type of any pair in an indiscernible has this property.
    Now assume that $p \in S_2(\emptyset)$ is such that for all $(a,b) \in \mathrm{im}(p)$, either $(a,a) \in \mathrm{im}(p)$ or $(b,b) \in \mathrm{im}(p)$.
    We will show that for any $n$, there are $f_1,\dots,f_n \in M_{[0,1]}$ with $\mathrm{tp}(f_i,f_j) = p$ for all $i < j$, so by compactness, in some elementary extension, there is an infinite sequence
    $(f_i : i \in I)$ such that for all $i < j$, $\mathrm{tp}(f_i,f_j) = p$.
    By Lemma \ref{lem_type_two}, $(f_i : i \in I)$ is indiscernible.

    As a consequence of \cite[Theorem 3.2]{no_trans}, there are $f,g \in M_{[0,1]}$ such that $\mathrm{tp}(f,g) = p$ and $\frac{f(t) + g(t)}{2} = t$ is the identity.
    Thus we may partition $[0,1]$ into three disjoint sets, $A_-,A_0,A_+$, on which $f - g$ is respectively negative, 0, and positive, and note that $A_-, A_+$ are open while $A_+$ is closed.
    Thus also $A_-$ and $A_+$ each consist of a countable number of open interval connected components.
    
    We will define our functions $f_i$ on $A_0$ and connected components of $A_+,A_-$ separately, and we will define them so that $\frac{1}{n}\sum_{i = 1}^n f_i(t) = t$ for all $t$.
    For each $1 \leq i \leq n$, if $t \in A_0$, then we let $f_i(t) = t$.
    Now let $(a,b)$ be a connected component of $A_+$, and we will define $f_1,\dots,f_n$ on $[a,b]$.
    As $f - g$ is positive on $(a,b)$, we see that $a = g(a) \leq g\left(\frac{a + b}{2}\right) < f\left(\frac{a + b}{2}\right) \leq f(b) = b$.
    At least one of $f\left(\frac{a + b}{2}\right), g\left(\frac{a + b}{2}\right)$ is in $A_0$, and both are in $[a,b]$, so it must be either $a$ or $b$.
    However, these numbers add to $a + b$, so they must be $b$ and $a$ respectively.
    By continuity and monotonicity, we see that the other values of $(c,d) \in \mathrm{im}(p)$ with $\frac{c + d}{2} \in [a,b]$ are exactly the points of the form $(t,a),(b,t)$ for $t \in [a,b]$.
    Thus the values of $(f_1,\dots,f_n)$ on $[a,b]$ should all be of the form $(b,\dots,b,t,a,\dots,a)$ for $t \in [a,b]$.

    It will thus suffice to define $f_1,\dots,f_n$ on $[a,b]$ such that
    \begin{itemize}
        \item for all $i$, $f_i$ is continuous and monotone on $[a,b]$,
        \item for all $i$, $f_i(a) = a$ and $f_i(b) = b$,
        \item for all $i < j$, $t \in [a,b]$, either $f_i(t) = b$ or $f_j(t) = a$.
    \end{itemize}
    
    We define our functions on $[a,b]$ by breaking up $[a,b]$ into $n$ subintervals of the form $\left[\frac{ia + (n-i)b}{n}, \frac{(i + 1)a + (n - i - 1)b}{n}\right]$,
    where $f_i(t) = a$ on $\left[a, \frac{ia + (n - i)b}{n}\right],$ $f_i(t)$ increases from $a$ to $b$ linearly on $\left[\frac{ia + (n-i)b}{n}, \frac{(i + 1)a + (n - i - 1)b}{n}\right]$,
    and $f_i(t) = b$ on $\left[\frac{(i + 1)a + (n - i - 1)b}{n}, b\right]$.
    If $i < j$, we see that for all $t \in [a,b]$, either $f_i(t) = b$ or $f_j(t) = a$, so we are done.

    If instead $(a,b)$ is a connected component of $A_-$, the functions can be defined similarly.
    As we have defined continuous, monotone functions on closed intervals covering $[0,1]$ in a way that endpoints agree and any pair $(f_i,f_j)$ with $i < j$ only takes values in $\mathrm{im}(p)$, we are done.
\end{proof}

\subsection{Another Language}
We now propose a new language for this structure.
Because by Lemma \ref{lem_type_two}, all types are determined by their restrictions to pairs of variables, it suffices to choose predicate symbols that generate all definable predicates on two variables.
By Stone-Weierstrass, it suffices to find a set of definable predicates on two variables that separates points on the type space $S_2$.
For this, we may take the family $\{\phi_\alpha(x,y) : \alpha \in [0,1] \cap \Q\}$, where when $(a,b) \in \mathrm{im}(\mathrm{tp}(f,g))$ is the unique point such that $a + b = \alpha$, $\phi_\alpha(f,g) = a$.
It is clear that each of these is 1-Lipschitz, so define $\mathcal{L}$ to be the language consisting only of 1-Lipschitz binary predicates $\phi_\alpha(x,y)$ for $\alpha \in [0,1] \cap \Q$.
Because the image of a type, and thus the type itself, is determined entirely by the value of these atomic predicates, $\mathrm{Th}(M_{[0,1]})$ eliminates quantifiers in $\mathcal{L}$.

We can also axiomatize $\mathrm{Th}(M_{[0,1]})$ fairly easily in this language.
For simplicity, we extend the language by quantifier-free definitions to include $\phi_\alpha(x,y)$ for $\alpha \in [0,1]$ by taking uniform limits.

\begin{lem}\label{lem_axiom}
    The theory of $M_{[0,1]}$ is axiomatized by the following theory, which we describe with equations and inequalities for clarity:
    \begin{align*}
       & \{\phi_\alpha(x,y) + \phi_\alpha(y,x) = \alpha : \alpha \in [0,1] \} \\
        \cup & \{\phi_\alpha(x,y) \leq \phi_\beta(x,y) : 0 \leq \alpha < \beta \leq 1 \} \\
        \cup & \{\inf_{x_1,\dots,x_n}\bigvee_{0 \leq k \leq m, i \neq j} |\phi_{c_k(i) + c_k(j)}(x_i,x_j) - c_k(i)| = 0 : c_0,\dots,c_m \in [0,1]^n \textrm{is a finite chain} \} 
    \end{align*}
\end{lem}
\begin{proof}
It suffices to require that for each pair of variables $x,y$, the set $\{(\phi_\alpha(x,y),\phi_\alpha(y,x)) : \alpha \in [0,1]\}$ forms a valid type in $S_2$,
and to require that every type in each $S_n$ is realized, which they are in all structures as the theory is separably categorical.
First, we require that $\phi_\alpha(x,y) + \phi_\alpha(y,x) = \alpha$ with an axiom for each $\alpha$.
To check that $\{(\phi_\alpha(x,y),\phi_\alpha(y,x)) : \alpha \in [0,1]\}$ is a chain from $0$ to $1$, we add axioms ensuring that $\phi_\alpha(x,y) \leq \phi_\beta(x,y)$ for each $\alpha < \beta$.
These also imply that $\{(\phi_\alpha(x,y),\phi_\alpha(y,x)) : \alpha \in [0,1]\}$ is connected, as the function $\alpha \mapsto (\phi_\alpha(x,y),\phi_\alpha(y,x))$ is 1-Lipschitz and thus continuous.

Now to ensure that each type is realized.
For each connected chain $C \subseteq [0,1]^n$ from 0 to 1, and each $m \in \N$, let $c_0,\dots,c_m$ be the points on $C$ such that at $c_i$, $\sum_{i = 1}^n x_i = \frac{in}{m}$.
Then let $(a_1,\dots,a_n)$ be such that $\{c_0,\dots,c_m\} \subseteq \mathrm{im}(a_1,\dots,a_n)$.
For each $c \in \mathrm{im}(a_1,\dots,a_n)$, there is some $i$ with $c_i \leq c \leq c_{i + 1}$, so $d(c,C) \leq d(c,c_i) \leq \frac{n}{m}$.
Thus if 
$$\bigvee_{0 \leq k \leq m, i \neq j} |\phi_{c_k(i) + c_k(j)}(x_i,x_j) - c_k(i)| = 0$$
 at a particular $(a_1,\dots,a_n)$, we find that $d(\mathrm{tp}(a_1,\dots,a_n),C) \leq \frac{n}{m}$ and if 
$$\inf_{x_1,\dots,x_n}\bigvee_{0 \leq k \leq m, i \neq j} |\phi_{c_k(i) + c_k(j)}(x_i,x_j) - c_k(i)| = 0$$
 for $(c_0,\dots,c_m)$ for all $m$, the type with image $C$ is realized.
As $\{c_0,\dots,c_m\}$ could be any chain from 0 to 1, and in fact this predicate will still be 0 for any finite chain, we simply require this for all finite chains.
\end{proof}

We now consider distal cell decompositions in this language.
\begin{defn}
If $\phi(x;y)$ is a definable predicate, and $\Psi$ is a finite set of definable predicates of the form $\psi(x;y_1,\dots,y_k)$,
then $\Psi$ \emph{weakly defines} a $\varepsilon$-\emph{distal cell decomposition} over $M$ for $\phi(x;y)$ when
for every finite $B \subseteq M^y$ with $|B| \geq 2$ and every $a \in M^x$, there are $\psi \in \Psi$ and $b_1,\dots,b_k \in M^x$ such that
$\psi(a;b_1,\dots,b_k) > 0$ and for all $a' \in M^x$,
$\psi(a';b_1,\dots,b_k) > 0$ implies $|\phi(a;b) - \phi(a';b)| \leq \varepsilon$ for all $b \in B$.
\end{defn}

\begin{thm}
Each $\phi_\alpha(x;y)$ admits a $\varepsilon$-distal cell decomposition over $M_{[0,1]}$ for each $\varepsilon > 0$, which we construct explicitly.
\end{thm}
\begin{proof}
Let $B \subseteq M_{[0,1]}$ be finite with $|B| \geq 2$.

For each $0 \leq i \leq n$, let $F_{i-}$ be a continuous function with support $\left[0,\frac{i}{n}\right)$, and let $F_{i+}$ be a continuous function with support $\left(\frac{i}{n},1\right]$.
We will show for each $0 < i < n$, there are there are some $\psi_{i-}(x),\psi_{i+}(x)$, with $\psi_{i-}(x)$ of the form either $1$ or $F_{i-}(\phi_\alpha(x;b_-))$,
and $\psi_{i+}(x)$ either of the form 1 or $F_{i+}(\phi_\alpha(x;b_+))$ with $b_-,b_+ \in B$,
such that $\psi_{i\pm}(a) > 0$, while $\psi_{i-}(a') > 0$ implies $\phi_\alpha(a';b) < \frac{i}{n}$ for each $b \in B$ with $\phi_\alpha(a';b) < \frac{i}{n}$, and 
$\psi_{i+}(a') > 0$ implies $\phi_\alpha(a';b) > \frac{i}{n}$ for each $b \in B$ with $\phi_\alpha(a';b) > \frac{i}{n}$.
Once we know this, we can let $\psi(x) = \bigwedge_{i = 0}^n \left(\psi_{i-}(x) \wedge \psi_{i+}(x)\right)$.
Then $\psi(x)$ will be of the form $\psi(x;b_1,\dots,b_k)$, where $\psi$ is one of a finite set $\Psi$ of formulas, and $b_1,\dots,b_k$.
This set $\Psi$ weakly defines a $\frac{2}{n}$-distal cell decomposition, because 
$\psi(a) > 0$, and for every $b \in B$, there is some $i$ such that $\frac{i}{n} < \phi_\alpha(a;b) < \frac{i + 2}{n}$,
so $\psi(a';b_1,\dots,b_k) > 0$ implies $\frac{i}{n} < \phi_\alpha(a';b) < \frac{i + 2}{n}$,
and thus $|\phi_\alpha(a;b) - \phi_\alpha(a';b)|\leq \frac{2}{n}$.

By symmetry, it suffices to construct $\psi_{i-}(x)$.
Let $b_- \in B$ maximize $\sup(b_-^{-1}(\{\alpha - \frac{i}{n}\}))$ under the constraint that $\phi_\alpha(a;b) < \frac{i}{n}$.
If there is not some $b_-$ satisfying this constraint, then we simply let $\psi_{i-}(x) = 1$, the rest of the requirements are trivial.
If it does exist, then we let $\psi_{i-}(a) = F_{i-}(\phi_\alpha(a;b_-))$, and by construction, $F_{i-}(\phi_\alpha(a;b_-)) > 0$.

If $a' \in M^x, b \in B$, we claim that $\phi_\alpha(a';b) < \frac{i}{n}$ if and only if the interval $b^{-1}((\alpha - \frac{i}{n},1])$ intersects the interval $a'^{-1}([0,\frac{i}{n}))$.
Let $t_\alpha$ be such that $a'(t_\alpha) + b(t_\alpha) = \alpha$.
If $\phi_\alpha(a';b) < \frac{i}{n}$, then these intervals intersect at $t_\alpha$.
If these intervals intersect at some $t$, then we know that $a'(t) < \frac{i}{n}$ and $b(t) > \alpha - \frac{i}{n}$.
If $a'(t) + b(t) < \alpha$, then $t < t_\alpha$, 
and thus $b(t_\alpha) \geq b(t) > \alpha - \frac{i}{n}$, so $\phi_\alpha(a';b) <\frac{i}{n}$, and similarly if $a'(t) + b(t) > \alpha$, then $t > t_\alpha$,
so $\phi_\alpha(a';b) = a'(t_\alpha) < \frac{i}{n}$. If $a'(t) + b(t) = \alpha$, then $\phi_\alpha(a';b) = a'(t) < \frac{i}{n}$.

Now assume $\psi_{i-}(a') > 0$, $\phi_\alpha(a;b) < \frac{i}{n}$, - we wish to show that $\phi_\alpha(a';b) < \frac{i}{n}$.
Because $\psi_{i-}(a') > 0$, we see that $b_-^{-1}((\alpha - \frac{i}{n},1])$ intersects $a'^{-1}([0,\frac{i}{n}))$,
and by the definition of $b_-$, because $\phi_\alpha(a;b) < \frac{i}{n}$, we know that $b_-^{-1}((\alpha - \frac{i}{n},1]) \subseteq b^{-1}((\alpha - \frac{i}{n},1])$.
Thus $b^{-1}((\alpha - \frac{i}{n},1])$ intersects $a'^{-1}([0,\frac{i}{n}))$ and $\phi_\alpha(a';b) < \frac{i}{n}$.
\end{proof}

\subsection{Nondiscreteness}
Dual linear continua provide the best example of distal metric structures that are truly different from distal discrete structures.
There are several possible criteria for determining whether a metric structure is non-discrete, and \cite{hansonsimple} compares several of these.
Of these, the strongest, there denoted as $\star$, is defined as follows:
\begin{defn}
    A metric structure has the property $\star$ when for any small partial type $\Sigma(x)$ (in finitely many variables), the metric space of realizations of $\Sigma(x)$ in the monster has a bounded number of connected components.
\end{defn}

\begin{thm}
    Dual linear continua have property $\star$.
\end{thm}
\begin{proof}
    By Theorem \ref{thm_lin_cont}, the monster model is isomorphic to $M_L$ for some linear continuum with endpoints $L$ - we shall assume that it is indeed $M_L$.
    By \cite[Theorem 3.1]{hansonsimple}, to check $\star$ it suffices to check that the space of realizations of complete types over small models are connected.
    Thus let $M \preceq M_L$ be a small model and let $p(x_1,\dots,x_n)$ be a complete $M$-type.
    There are unique functions $h_1,\dots,h_n : [0,1] \to [0,1]$ such that for any realization $(f_1,\dots,f_n)$ of $p$, with $f = \frac{1}{n}\sum_{i = 1}^n f_i$,
    we have $f_i = h_i \circ f$.
    There is also a unique complete $M$-type $q(x)$ of averages of realizations of $p$.
    We see that $f \mapsto (h_1,\dots,h_n) \circ f$ is a function from the space of realizations of $q$ to the space of realizations of $p$,
    and is continuous with respect to the $\sup$ metric, so it suffices to show that the space of realizations of $q$, in one variable, is connected.
    In fact, we will show that it is convex, and thus path-connected.

    Suppose $f,g \in M_L$ are both realizations of $q$. It suffices to show that for $\lambda \in [0,1]$, $\mathrm{tp}((1 - \lambda)f + \lambda g/M) = q$.
    In fact, by Lemma \ref{lem_type_two}, it suffices to check that for each $a \in M$, $\mathrm{tp}((1 - \lambda)f + \lambda g,a)$, or equivalently $\mathrm{im}((1 - \lambda)f + \lambda g,a)$, does not depend on $\lambda$.
    For each $c \in [0,1]$, both $f$ and $g$ obtain the same closed interval of values on the preimage $a^{-1}(\{c\})$.
    Thus for any $t \in a^{-1}(\{c\})$, we have that $(1 - \lambda)f(t) + \lambda g(t)$ is also in this interval, so $((1 - \lambda)f(t) + \lambda g(t),a(t)) \in \mathrm{im}((f,a))$,
    implying that $\mathrm{im}((1 - \lambda)f+ \lambda g,a) = \mathrm{im}((f,a))$ for all $\lambda$.
\end{proof}

\section{Nonexamples}\label{sec_non}
In discrete logic, there is an open question as to which NIP structures admit distal expansions.
The Strong Erd\H{o}s-Hajnal property is one requirement for admitting a distal expansion, and we have shown that this is still required in continuous logic,
but little else is known.
In continuous logic, however, we can see a wide class of NIP structures which cannot admit distal expansions for a seemingly different reason: Banach structures.
We thank James Hanson for pointing this out.

\begin{defn}
A \emph{Banach structure} is an expansion of a Banach space, viewed as a metric structure. The theory of a Banach structure is called a \emph{Banach theory}.
\end{defn}

It will be easy to see that many of these are not distal, because of the following fact:

\begin{fact}[{\cite[Corollary 6.10]{hanson_smsets}}]
Every Banach theory with infinite dimensional models has an infinite indiscernible set in some model.
\end{fact}

\begin{cor}
No Banach theory with infinite dimensional models is distal.
\end{cor}
\begin{proof}
This is true because no distal structure has an infinite indiscernible set, by the same proof as in discrete logic:

If it did, we could partition such a set into two infinite subsets $I, J$ and an extra element, $d$.
Then by the indiscernibilty of the overall set, $I + J$ is indiscernible over $d$, and thus by distality, $I + d + J$ is indiscernible over $d$, implying that every element of $I + d + J$ satisfies $x = d$.
This clearly contradicts the set being infinite.
\end{proof}

One particularly interesting class of Banach structures is randomizations.
If $T$ is a metric theory, then the \emph{randomization theory} $T^R$ of $T$ can be constructed in a few ways, each of which captures the idea that a model of $T^R$ consists of random variables valued in models of $T$.
The construction in \cite{randvar} adds to the sorts of $T$ an extra sort, consisting of an algebra of random variables, which is an $L_1$-space, and thus is a Banach structure.
While the randomization of a stable theory is stable (\cite[Theorem 4.9]{randvar}) and the randomization of an NIP theory is NIP (\cite[Theorem 5.3]{randomVC}), we see that the same is not true of distality, as the randomization of \emph{any} structure is not distal.
Restricting to the original sorts of $T$ will not change this, as the random variable sort is interpretable from the induced structure on the other sorts.

\bibliographystyle{plainurl}
\bibliography{ref.bib}

\begin{thebibliography}{10}

\bibitem{anderson2}
Aaron Anderson.
\newblock Generically stable measures and distal regularity in continuous
  logic.
\newblock Preprint, 2023.
\newblock \href {http://arxiv.org/abs/2310.04393} {\path{arXiv:2310.04393}}.

\bibitem{anderson1}
Aaron Anderson.
\newblock {NIP} and distal metric structures.
\newblock Preprint, 2025.
\newblock \href {http://arxiv.org/abs/2310.04393} {\path{arXiv:2310.04393}}.

\bibitem{distal_val}
Matthias Aschenbrenner, Artem Chernikov, Allen Gehret, and Martin Ziegler.
\newblock Distality in valued fields and related structures.
\newblock {\em Transactions of the American Mathematical Society},
  375(7):4641--4710, 2022.

\bibitem{randomVC}
Ita\"i Ben~Yaacov.
\newblock Continuous and random {Vapnik-Chervonenkis} classes.
\newblock {\em Israel Journal of Mathematics}, 173(1):309--333, September 2009.

\bibitem{defOSgrp}
Ita{\"\i} Ben~Yaacov.
\newblock Definability of groups in $\aleph_0$-stable metric structures.
\newblock {\em The Journal of Symbolic Logic}, 75(3):817--840, 2010.

\bibitem{randvar}
Ita{\"\i} Ben~Yaacov.
\newblock On theories of random variables.
\newblock {\em Israel Journal of Mathematics}, 194:957--1012, 2013.

\bibitem{mvf}
Ita{\"\i} Ben~Yaacov.
\newblock Model theoretic properties of metric valued fields.
\newblock {\em The Journal of Symbolic Logic}, 79(3):655--675, 2014.

\bibitem{no_trans}
Ita{\"\i} Ben~Yaacov.
\newblock On a {R}oelcke-precompact {P}olish group that cannot act transitively
  on a complete metric space.
\newblock {\em Israel Journal of Mathematics}, 224(1):105--132, 2018.

\bibitem{mtfms}
Ita\"i Ben~Yaacov, Alexander Berenstein, C.~Ward Henson, and Alexander
  Usvyatsov.
\newblock Model theory for metric structures.
\newblock {\em London Mathematical Society Lecture Note Series}, 350:315, 2008.

\bibitem{random09}
Ita{\"\i} Ben~Yaacov and H.~Jerome Keisler.
\newblock Randomizations of models as metric structures.
\newblock {\em Confluentes Mathematici}, 1(02):197--223, 2009.

\bibitem{rtrees}
Sylvia Carlisle and C~Ward Henson.
\newblock Model theory of {R}-trees.
\newblock {\em Journal of Logic and Analysis}, 12, 2020.

\bibitem{distal_reg}
Artem Chernikov and Sergei Starchenko.
\newblock Regularity lemma for distal structures.
\newblock {\em Journal of the European Mathematical Society},
  20(10):2437--2466, 2018.

\bibitem{urysohn}
Gabriel Conant and Caroline Terry.
\newblock Model theoretic properties of the urysohn sphere.
\newblock {\em Annals of Pure and Applied Logic}, 167(1):49--72, 2016.

\bibitem{hanson_approx}
James Hanson.
\newblock Approximate isomorphism of metric structures.
\newblock {\em Mathematical Logic Quarterly}, 69(4):482--507, 2023.

\bibitem{hansonsimple}
James Hanson.
\newblock A simple continuous theory.
\newblock Preprint, 2023.
\newblock \href {http://arxiv.org/abs/2306.14324} {\path{arXiv:2306.14324}}.

\bibitem{hanson_smsets}
James Hanson.
\newblock Strongly minimal sets and categoricity in continuous logic.
\newblock {\em Memoirs of the American Mathematical Society}, 311(1574), 2025.

\bibitem{ibarlucia1}
Tom{\'a}s Ibarluc{\'\i}a.
\newblock The dynamical hierarchy for {R}oelcke precompact {P}olish groups.
\newblock {\em Israel Journal of Mathematics}, 215(2):965--1009, 2016.

\bibitem{distal_simon}
Pierre Simon.
\newblock Distal and non-distal {NIP} theories.
\newblock {\em Annals of Pure and Applied Logic}, 164(3):294--318, 2013.

\bibitem{vdd_tame}
Lou Van~den Dries.
\newblock {\em Tame topology and o-minimal structures}, volume 248.
\newblock Cambridge university press, 1998.

\end{thebibliography}

\end{document}